\documentclass[11pt,leqno,oneside,letterpaper]{amsart}
\usepackage{amssymb,amsmath,amsfonts,latexsym, amscd, amsthm, graphicx, url, epsfig, mathtools, algorithm, algpseudocode}
\usepackage[letterpaper,left=1.75truein, top=1truein]{geometry}
\usepackage{color}
\usepackage[colorlinks,linkcolor=blue,citecolor=blue]{hyperref}
\newtheorem{theorem}{Theorem}[section]
\newtheorem{lemma}[theorem]{Lemma}

\theoremstyle{definition}
\newtheorem{remark}[theorem]{Remark} 

\newtheorem{example}[theorem]{Example} 
 
\newtheorem{question}[theorem]{Question}

\newtheoremstyle{cases}
  {12pt plus 6 pt}
  {2pt}
  {\bfseries}   
  {}
  {\bfseries}
  {.}
  {.5em}
  {}
\theoremstyle{cases}

\numberwithin{subcase}{case} \numberwithin{subsubcase}{subcase}
\numberwithin{equation}{subsection}

\begin{document}

\title[Bi-orderability of knot groups]{Testing bi-orderability of knot groups}

\author[Adam Clay]{Adam Clay and Colin Desmarais and Patrick Naylor}
\thanks{Adam Clay was partially supported by NSERC grant RGPIN-2014-05465}
\address{Department of Mathematics, 342 Machray Hall, University of Manitoba, Winnipeg, MB, R3T 2N2.}
\email{Adam.Clay@umanitoba.ca, naylorp@myumanitoba.ca, umdesmac@myumanitoba.ca}
\urladdr{http://server.math.umanitoba.ca/~claya/}

\begin{abstract}
We investigate the bi-orderability of two-bridge knot groups and the groups of knots with 12 or fewer crossings by applying recent theorems of Chiswell, Glass and Wilson \cite{CGW14}.  Amongst all knots with 12 or fewer crossings (of which there are 2977), previous theorems were only able to determine bi-orderability of 599 of the corresponding knot groups.  With our methods we are able to deal with 191 more.
\end{abstract}
\maketitle
\vspace{-.6cm}
\begin{center}
\today
\end{center}

\section{Introduction}
A group $G$ is called \textit{bi-orderable} if there exists a strict total ordering of the elements of $G$ such that $g<h$ implies $fg<fh$ and $gf<hf$ for all $f, g, h \in G$.

Several authors over the past few years have investigated the question of bi-orderability of $3$-manifold groups and knot groups \cite{PR03, PR06, LRR08, CR12, Ito13a, CGW14}.  With the exception of \cite{CGW14}, all the theorems developed apply only to $3$-manifolds, or specifically to knot complements, which fibre over $S^1$.  As a test of the effectiveness of available theorems, the authors of \cite{CR12} computationally investigated the knot groups for knots with twelve or fewer crossings appearing in the table available from \textit{KnotInfo} \cite{CL}.  They found that of of the 2977 knots with twelve or fewer crossings, 1246 of them fibre over $S^1$ and so the available theorems were applicable.  From these 1246, only 12 were found to have bi-orderable knot group,  and 487 were found to have non-bi-orderable knot group, which we could interpret as a success rate of about $\frac{12+487}{2977} \sim 17 \%$.

This paper extends the efforts of \cite{CR12} to the nonfibred case by using the newly available theorems of \cite{CGW14}.  The authors of  \cite{CGW14} have already observed that their methods can be used to show that the knot group of $5_2$ is not bi-orderable\footnote{Naylor and Rolfsen showed independently via a computer program that this group is not bi-orderable.  In fact, their program showed that the group admits generalized torsion \cite{NR14}.} ($5_2$ is a non-fibred knot), and they produce families of Alexander polynomials of degree four for which the corresponding knots will always have non-bi-orderable knot group.

As the theorems of \cite{CGW14} apply only to groups with two generators and one relator, we focus our attention on two sources of knot groups having these properties.  Our first family of such knot groups comes from two-bridge knots, whose groups are known to have a presentation with two generators and one relator of the form $\langle a, b | aw =wb \rangle$ (see Section \ref{twobridge_section}).  The second family is all knots from the the \textit{KnotInfo} table having twelve or fewer crossings, bridge number greater than two, and a group presentation (as computed by SnapPy \cite{SnapPy}) with two generators and one relator. 

With these techniques we were able to determine orderability of many two-bridge knot groups, all twist knot groups, and find 6 new bi-orderable knot groups and 185 new non-bi-orderable knot groups arising from knots with fewer than 12 crossings.  Despite these advances, at the time of this writing the first knot appearing in the tables for which bi-orderablity of the knot group cannot be determined is the knot $6_2$.  

The difficulty in the case of $6_2$ is that its Alexander polynomial has two real roots and two complex roots, and the available theorems are only applicable when all the roots of the Alexander polynomial are real or all the roots are complex.  This difficulty highlights a problem that is made explicit in \cite[Corollary 12]{NR14}, namely that the Alexander polynomial alone is not enough to detect non-bi-orderability of knot groups.  As such, one may not be able to address the case of $6_2$ by strengthening existing theorems, as they all depend on the Alexander polynomial.

\begin{question}
Can non-bi-orderability of the knot group be determined by examining knot invariants other than the Alexander polynomial?
\end{question}

\textbf{Acknowledgments}.  We would like to thank Nathan Dunfield for providing us with the a list of knot group presentations, as well as the code he used to compute them using SnapPy.


\section{Background}

For convenience of the reader, we state the results of \cite{CGW14} that we will use here.  We write $b^a$ in place of $a^{-1}b a$, and for a word $w \in F(a, b)$ in the free group on generators $a$ and $b$, write $w_b$ and $w_a$ for the total exponent sum (which we will call the weight) of $b$ and $a$ in the word $w$.   Given such a word $w$, if $w_a =0$ then we can rewrite $w$ as
\[ w = b^{m_1 a^{d_1}} \cdots b^{m_r a^{d_r}}
\]
for some integers $m_i, d_i$ and $r \geq 1$.  For all $j \in \mathbb{Z}$, set $\tau_j (w) = \{ i : d_i =j \}$ and let $S_w = \{ j : \sum_{i \in \tau_j(w)} m_i \neq 0 \}$.  

We say that a word $w$ of the form above is \textit{tidy} if $\tau_j(w) = \emptyset$ for all $j$ satisfying either  $j > \max \{S_w \}$ or $j < \min \{ S_w \}$.  Set $\ell =  \max \{S_w \}$; we say that $w$ is \textit{principal} if it is tidy and $|\tau_{\ell}(w)| =1$.   In the case that $w$ is principal  and $\tau_{\ell}(w) = \{k\}$, we call $w$ \textit{monic} if in addition $m_k =1$.  Set  $s= \min \{ S_w \}$, when $\pi_1(S^3 \setminus K) = \langle a, b | w \rangle$ with $w$ as above, the Alexander polynomial has formula $\Delta_K(t) = \sum_{i=1}^r m_i t^{d_i - s}$.  We can group like powers and rewrite this as 
\begin{equation} \label{CGWeqn} \Delta_K(t) = \sum_{j \in \mathbb{Z}} \left( \sum_{i \in \tau_j(w)} m_i \right)t^{j-s} \tag{*}
\end{equation}
where we understand that the coefficient of $t^{j-s}$ is zero when $\tau_j(w) = \emptyset$.

\begin{theorem}[\cite{CGW14}, Corollary 2.5]
\label{CGW_theorem}
Let $K$ be a knot in $S^3$, and suppose that $\pi_1(S^3 \setminus K)$ has a presentation of the form $ \langle a, b | w \rangle$ where $w$ is tidy.  Let $\Delta_K(t)$ denote the Alexander polynomial of $K$.  Then:
\begin{enumerate}
\item If $\pi_1(S^3 \setminus K)$ is bi-orderable, then  $\Delta_K(t)$ has a positive real root.

\item If $w$ is monic and all the roots of  $\Delta_K(t)$ are real and positive, then $\pi_1(S^3 \setminus K)$ is bi-orderable.

\item If $w$ is principal,  $\Delta_K(t) = a_0 + \cdots + a_{d-1}t^{d-1} - m t^d$ where $\gcd \{ a_0, \ldots, a_{d-1} \} = 1$ and $a_{d-1}$ is not divisible by $m$, and all the roots of $\Delta_K(t)$ are real and positive, then $\pi_1(S^3 \setminus K)$ is bi-orderable.
\end{enumerate}
\end{theorem}

\section{Two bridge knots}
\label{twobridge_section}
Recall that according to Schubert, $2$-bridge knots are classified by coprime pairs of odd integers $p$ and $q$, with $0<p<q$.  Thus every two-bridge knot may be written as $K_{\frac{p}{q}}$ where $\frac{p}{q}$ is a reduced fraction.  The fundamental group has presentation
\[ \pi_1(S^3 \setminus K_{\frac{p}{q}}) = \langle a, b | aw=wb  \rangle  
\]
where $w= b^{\epsilon_1}a^{\epsilon_2} \cdots b^{\epsilon_{q-2}} a^{\epsilon_{q-1}}$ and $\epsilon_i = (-1)^{\lfloor \frac{ip}{q} \rfloor}$.  This formula follows from Schubert's normal form \cite{Schubert56}, see e.g. the discussion in \cite{Murasugi61}. 

\begin{lemma}
\label{q_lemma}
If $\Delta_{K_{\frac{q}{p}}}(t) =a_0 + a_1 t + \cdots + a_nt^n$, then $\Delta_{K_{\frac{q}{p}}}(-1)= \sum_{i=1}^n |a_i| =q$.
\end{lemma}
\begin{proof}
From the group presentation above, one can use Fox calculus to compute (see \cite{HJ13}):
\[ \Delta_{K_{\frac{p}{q}}}(t) = 1 - t^{\epsilon_1} + t^{\epsilon_1 + \epsilon_2} - t^{\epsilon_1 + \epsilon_2 + \epsilon_2} + \cdots + t^{\epsilon_1 + \cdots + \epsilon_{q-1}}
\]
Since $\epsilon_i = \pm 1$, $\sum_{i=1}^{\ell} \epsilon_i$ is odd if and only if $\ell$ is odd.  From this we draw two conclusions: First, if we regroup terms and write $\Delta_{K_{\frac{q}{p}}}(t) =a_0 + a_1 t + \cdots + a_nt^n$, then $a_i$ is negative if and only if $i$ is odd.  Therefore $\Delta_{K_{\frac{q}{p}}}(-1)= \sum_{i=1}^n |a_i|$.  Second, we may compute from the above formula that 
 \[  \Delta_{K_{\frac{p}{q}}}(-1) = \underbrace{1+1+ \cdots +1 }_{\text{~$q$ times}} = q.
\]
\end{proof}

\begin{lemma}
\label{twobridge_lemma}
Every two bridge knot group admits a presentation of the form $\langle x, y | r \rangle$ where the relator $r$ is a tidy word in the generators $x,y$.
\end{lemma}
\begin{proof}
Consider an arbitrary two bridge knot $K_{\frac{p}{q}}$, and the presentation of its knot group with notation as defined above
\[ \pi_1(S^3 \setminus K_{\frac{q}{p}}) = \langle a, b | aw=wb  \rangle  
\]
Define a homomorphism $\phi: F(a,b) \rightarrow F(x,y)$  by $\phi(a) =xy$, $\phi(b) =y$.  This descends to an isomorphism of the group  $ \langle a, b | aw=wb  \rangle  $ with the group presented by $\langle x, y | R \rangle$ where $R = xy \phi(w) y^{-1} \phi(w)^{-1}$, note that $y$ has weight zero in $R$.  We claim that $R$ is a tidy word in $\{x, y \}$.

Observe that that since each occurence of either $a$ or $a^{-1}$ in the word $awb^{-1}w^{-1}$ results in exactly one occurence of $x$ or $x^{-1}$ in $R$, the letters $x$ and $x^{-1}$ occur a total of $q$ times in $R$. Thus upon rewriting $R$ in the form
\[ x^{m_1 y^{d_1}} \cdots x^{m_r y^{d_r}}
\]
we have $\sum_{i=1}^r |m_i| =q$, since the rewriting is accomplished without cancelling any powers of $x$.

Now suppose that $R$ is not tidy, so there exists $j_0$ such that $\tau_{j_0}(R)  \neq \emptyset$ and $j_0 \notin S_R$, and therefore $\sum_{i \in \tau_{j_0}(w)} m_i  =0$.  
We compute
\[
\begin{array}{rcll}
 q & = &\displaystyle  \sum_{j \in \mathbb{Z}} \left| \sum_{i \in \tau_j(w)} m_i \right| & \mbox{by Lemma \ref{q_lemma} and (\ref{CGWeqn})} \\
 & = &\displaystyle  \sum_{j \neq j_0} \left| \sum_{i \in \tau_{j}(R)} m_i \right| &  \mbox{since $\sum_{i \in \tau_{j_0}(w)} m_i  =0$} \\
 & \leq & \displaystyle \sum_{
\substack{i \in \mathbb{Z}\\
i \notin \tau_{j_0}(R)}} |m_i| & \\
 &  < &\displaystyle  \sum_{i=1}^r |m_i| &  \hspace{15em} \\
& = & q
\end{array}
\]
and this contradiction completes the proof.
\end{proof}

As an immediate consequence, we can apply Theorem \ref{CGW_theorem}(1) to all two-bridge knot groups.

\begin{theorem}
\label{two_bridge_theorem}
Suppose that $K$ is a two-bridge knot with Alexander polynomial $\Delta_K(t)$.  If $\pi_1(S^3 \setminus K)$ is bi-orderable, then $\Delta_K(t)$ has a positive real root.
\end{theorem}

\section{Twist knots}

Twist knots are a subfamily of two-bridge knots whose diagrams appears as in Figure \ref{twist_figure}. Considered as a two-bridge knot $K_{\frac{p}{q}}$, the twist knot with $\frac{q-1}{2}$ twists ($q$ odd) corresponds to the case $p=q-2$.   We simplify notation by writing $K_q$ for the twist knot with $\frac{q-1}{2}$ twists, instead of $K_{\frac{q-2}{q}}$.
\vspace{1ex}
\begin{figure}[h!]
\setlength{\unitlength}{0.8cm}
\begin{picture}(3,2)(1,0)
\put(0,0){
\includegraphics[scale=0.35]{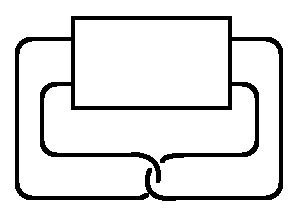}}
\put(1.55,2.2){\small{$\frac{q-1}{2}$ twists}}
\end{picture}
\caption{The twist knot $K_q$}
\label{twist_figure}
\end{figure}
\begin{theorem}
\label{twist_theorem}
Let $q > 3$ be an odd integer.  If $q  \equiv 1 \pmod{4}$ then $\pi_1(S^3 \setminus K_q)$ is bi-orderable, and if  $q  \equiv 3 \pmod{4}$ then $\pi_1(S^3 \setminus K_q)$ is not bi-orderable.
\end{theorem}

\begin{proof} 

The Alexander polynomial of the twist knot with $\frac{q-1}{2}$ twists is \cite[p. 80]{SN99}:
\[
    \Delta(t) = 
\begin{cases}
    \frac{1-q}{4} + \left(\frac{q+1}{2} \right) t + \left( \frac{1-q}{4} \right) t^2& \text{if } q \equiv 1 \pmod{4} \\
    \frac{1+q}{4} + \left(\frac{1-q}{2} \right) t + \left( \frac{1+q}{4} \right) t^2& \text{if } q \equiv 3 \pmod{4}
\end{cases}
\]

One can check that when $q \equiv 3 \pmod{4}$, the Alexander polynomial has no real roots and thus $\pi_1(S^3 \setminus K_q)$ is not bi-orderable by Theorem \ref{two_bridge_theorem}.  

On the other hand when  $q  \equiv 1 \pmod{4}$ both of the roots of the Alexander polynomial are real and positive, we will prove bi-orderability by applying Theorems \ref{CGW_theorem}(2) and \ref{CGW_theorem}(3).  First note that $\frac{q+1}{2} = -2\left( \frac{1-q}{4} \right)+1$, so $\frac{q+1}{2}$ and $\frac{1-q}{4}$ are relatively prime when $q >5$.  Thus when $q>5$ the coefficients of the Alexander polynomial satisfy the necessary divisibility conditions to apply Theorem \ref{CGW_theorem}(3), when $q=5$ we will apply Theorem  \ref{CGW_theorem}(2).  To finish the proof we show that $\pi_1(S^3 \setminus K_q)$ admits a presentation with two generators and a single principal relator which is also monic when $q=5$.  We begin with the standard two-bridge presentation $\pi_1(S^3 \setminus K_{\frac{p}{q}}) = \langle a, b | aw=wb  \rangle  $ and first determine the values of $\epsilon_i$ appearing in the formula for $w$.

For $i$ satisfying $1 \leq i \leq \frac{q-1}{2}$ we have $i > \frac{i(q-2)}{q} > i-1$ and so $\lfloor \frac{i(q-2)}{q} \rfloor = i-1$, whereas for $i$ satisfying $\frac{q+1}{2} <i < q-1$ we have $i-1 > \frac{i(q-2)}{q} > i-2$ and thus $\lfloor \frac{i(q-2)}{q} \rfloor = i-2$. This allows us to compute that 
\[ w = (\underbrace{ba^{-1}ba^{-1} \ldots ba^{-1}b}_{\frac{q-3}{2} \mbox{ letters }})(a^{-1}b^{-1})(\underbrace{ab^{-1} ab^{-1} \ldots ab^{-1}a}_{\frac{q-3}{2} \mbox{ letters }})
\]
Note that since $q  \equiv 1 \pmod{4}$ the segments of $\frac{q-3}{2}$ letters are of odd length (as we have written).

As in the proof of Lemma \ref{twobridge_lemma}, we will apply the homomorphism $\phi: F(a,b) \rightarrow F(x,y)$  with $\phi(a) =xy$, $\phi(b) =y$ to create a new presentation $\langle x, y | R \rangle$ of $\pi_1(S^3 \setminus K_q)$ with relator $R = xy \phi(w) y^{-1} \phi(w)^{-1}$. From the calculations above we find
\[\phi(w)=x^{-\frac{q-1}{4}}y^{-1}x^{\frac{q-1}{4}}y\]
and thus 
\[R=xyx^{-\frac{q-1}{4}}y^{-1}x^{\frac{q-1}{4}}y^{-1}x^{-\frac{q-1}{4}}yx^{\frac{q-1}{4}}\]
We replace $R$ with the relator $y^{-1} R y$ to get
\[ y^{-1}R y = (y^{-1}xy)(x^{-\frac{q-1}{4}})(y^{-1}x^{\frac{q-1}{4}}y)(y^{-2}x^{-\frac{q-1}{4}}y^2)(y^{-1}x^{\frac{q-1}{4}}y)
\]
from which we read off the sets \(\tau_j\) for \(j=0,1,2\), finding
\[\tau_0(y^{-1}R y)=\{2\},\tau_1(y^{-1}R y)=\{1,3,5\}, \tau_2(y^{-1}R y)=\{4\}\]
and $S_{y^{-1}Ry} = \{ 0, 1, 2\}$.
Since $|\tau_2(y^{-1}R y)| =1$, the relator $y^{-1} R y$ is principal.   Note that $\tau_2(y^{-1}R y)=\{4\}$ and $m_4 = \frac{q-1}{4}$, which is equal to one when $q=5$.  Thus the relator is also monic when $q=5$.
\end{proof}

\section{Computational results and knots with bridge number greater than two}

In order to apply Theorem \ref{CGW_theorem} to a group $G = \langle a, b | w \rangle$ one must find a presentation of $G$ where the relator $w$ has the form
\[ w = b^{m_1 a^{d_1}} \cdots b^{m_r a^{d_r}}
\]
Finding a presentation with a relator of this form is always possible when the group admits two generators and one relator, though there are possibly many different ways of doing so.  The key to our method is contained in the following lemma.

\begin{lemma} \cite[Chapter V, Lemma 11.8]{LS77}
Suppose that  $ G = \langle a, b | w \rangle$ and denote the weight of $a$ and $b$ in $w$ by $w_a, w_b$ respectively.  Assume without loss of generality that $0<w_a \leq w_b$.  Set  \[k =- \left\lfloor \cfrac{w_b}{w_a} \right\rfloor, \phi(a)=xy^k \mbox{ and } \phi(b)=y. \]
This defines an isomorphism $\phi: \langle a, b | w \rangle \rightarrow \langle x, y | \phi(w) \rangle$, moreover the weights $\phi(w)_x$ and $\phi(w)_y$ of $x$ and $y$ in $\phi(w)$ satisfy
\[ |\phi(w)_x| + |\phi(w)_y| < |w_a| + |w_b|
\]
\end{lemma}
\begin{proof}
The map $\phi$ is an isomorphism since it has inverse $\phi^{-1}(x) = ab^{-k}$, $\phi^{-1}(y) = b$.  Note that the weights $\phi(w)_x$ and $\phi(w)_y$ satisfy $\phi(w)_x = w_a$ and $\phi(w)_y = w_b+kw_a$.  In particular, there are bounds $0 \leq w_b+kw_a < w_b$ so that 
\[ |\phi(w)_x| + |\phi(w)_y|  =|\phi(w)_x| + |w_b+kw_a|  < |w_a| + |w_b|
\]
as claimed.
\end{proof}

The conclusion $ |\phi(w)_x| + |\phi(w)_y| < |w_a| + |w_b|$ means that upon iteratively applying this lemma to a group presentation (with appropriate variable changes made at each step to guarantee that the hypothesis $0<w_a \leq w_b$ is satisfied by the input) one will eventually produce a two generator one relator presentation of the group $G$ such that one of the generators has weight zero in the relator.  Supposing that the generator $a$ has weight zero, it is then possible to rewrite the relator in the form:
\[ w = b^{m_1 a^{d_1}} \cdots b^{m_r a^{d_r}}
\]
Pseudocode for the procedure is found in Appendix A.  The results of applying this algorithm to all knots with 12 or fewer crossings and checking for tidy, principal or monic resulting words are in Appendix B.

\begin{remark}
Note that in the lemma above, we could also have used the substitution  $\phi(a)=xy^k$ and $\phi(b)=xyx^{-1}$, with $k$ as above.  Therefore one can also iterate the substitution $\phi(a)=xy^k$ and $\phi(b)=xyx^{-1}$ in order to find a relator having weight zero in one of the generators, although in practice we found that this yielded few new results.  Indeed, SnapPy gives a presentation for the knot $9_{16}$ which has one non-tidy relator which becomes tidy upon iterating the substitution $\phi(a)=xy^k$ and $\phi(b)=xyx^{-1}$, but it does not become tidy by iterating $\phi(a)=xy^k$ and $\phi(b)=y$.  Of all knots with twelve or fewer crossings, this was the only instance where one algorithm yielded a different result than the other.
\end{remark}

\begin{example}
The knot $10_{52}$ is non-fibered, so none of the theorems available from \cite{PR03, PR06, LRR08, CR12} apply.  Moreover it has bridge index $3$, so Theorem \ref{two_bridge_theorem} does not apply either; it is the first knot in the tables for which the algorithm in Appendix A succeeds where these theorems fail. SnapPy gives the group presentation $\pi_1(S^3 \setminus 10_{52}) = \langle a, b | w \rangle$ where 
\[
\begin{split}
&\quad
w=ab^2a^2b^2aBab^2a^2bAB^2A^2B^2AbAB^2A^2Bab^2a^2b^2aBab^2a^2bAB^2A^2B^2Aba^2b^2 \hspace{10em}  \\
&\quad \hspace{9em} aBab^2a^2b^2aBA^2B^2AbAB^2A^2B^2Aba^2b^2aBab^2a^2b^2aBA^2B^2AbAB^2A^2B
\end{split}
\]
Here, we write capital letters in place of inverses in order to simplify notation.  Note that in the above word, neither generator has weight zero: we find $w_a = 3$ and $w_b=3$.  Applying the algorithm described in Appendix A yields the substitution $a \mapsto bA^3$ and $b \mapsto a^4B$, giving a new presentation of the form $ \langle a, b | w' \rangle$
where 
\[
\begin{split}
&\quad
w'=baBabaBabA^4baBabaBAbABAbABa^4BAbABAbaBabaBabA^4baBabaBAbABAbABab
 \hspace{10em}  \\
&\quad \hspace{1em} aBabA^4baBabaBabABAbABa^4BAbABAbABabaBabA^4baBabaBabABAbABa^4BAbABA \\
\end{split}
\]
The generator $a$ now has weight zero in $w'$.  After writing $w'$ in the form 
\[ w' = b^{m_1 a^{d_1}} \cdots b^{m_r a^{d_r}}
\]
and conjugating so that $\mathrm{min} \{ d_{1}, \ldots ,d_{r} \} = 0$, we find
\[
\begin{split}
&\quad
w'= b^{a^{4}}b^{-a^{3}}b^{a^{2}}b^{-a}bb^{a^{4}}b^{-a^{3}}b^{a^{2}}b^{-a}
b^{a^{2}}b^{-a^{3}}b^{a^{4}}b^{-a^{5}}b^{-a}b^{a^{2}}b^{-a^{3}}b^{a^{4}}b^{-a^{3}}
b^{a^{2}}b^{-a}bb^{a^{4}}b^{-a^{3}}b^{a^{2}}b^{-a}
 \hspace{10em}  \\
&\quad \hspace{2.5em} b^{a^{2}}b^{-a^{3}}
b^{a^{4}}b^{-a^{5}}b^{a^{4}}b^{-a^{3}}b^{a^{2}}b^{a^{6}}b^{-a^{5}}b^{a^{4}}b^{-a^{3}}
b^{a^{2}}b^{-a^{3}}b^{a^{4}}b^{-a^{5}}b^{-a}b^{a^{2}}b^{-a^{3}}b^{a^{4}}b^{-a^{5}}b^{-a^{3}}b^{a^{2}}
b^{a^{4}}\\
&\quad \hspace{2.5em}
b^{a^{6}}b^{-a^{5}}b^{a^{4}}b^{-a^{3}}b^{a^{2}}b^{-a^{3}}
b^{a^{4}}b^{-a^{5}}b^{-a}b^{a^{2}}b^{-a^{3}}
\end{split}
\]
One can easily check by hand from this expression that $\min S_{w'} = 0$ and $\max S_{w'} = 6$, and clearly $\tau_j$ is empty for $j>6$ and $j<0$.  Therefore the word $w'$ is tidy.
The Alexander polynomial of $10_{52}$ is 
\[ \Delta(t) = 2-7t+ 13t^2-15t^3+ 13t^4-7t^5+ 2t^6
\]
and one can numerically verify that all of its roots are complex.  Thus, by Theorem \ref{CGW_theorem}(1) the knot $10_{52}$ does not have a bi-orderable knot group. 
\end{example}

\appendix
\section{Finding a presentation having a relator with zero weight} 

\begin{algorithmic}
\Function{Zero\_weight}{String $w$}
  \State int $w_a$;
  \State int $w_b$;
  \While{weight of $a$ in $w$ is not 0} 
    \State $w_a$ = weight of $a$ in $w$;
    \State $w_b$ = weight of $b$ in $w$;
    \If{$w_a < 0$} 
      \State swap $a$'s and $a^{-1}$'s;
      \State $w_a = -w_a$;
     \EndIf
    \If{$w_b < 0$}
      \State swap $b$'s and $b^{-1}$'s;
      \State $w_b = -w_b$;
    \EndIf
    \If{$w_a > w_b$}
      \State swap $a$'s and $b$'s in $w$;
      \State swap $a^{-1}$'s and $b^{-1}$'s in $w$;
      \State swap $w_a$ and $w_b$;
    \EndIf
    \If{weight of $a$ in $w$ is not 0}
      \State replace $a$'s with $ab^{- \lfloor w_b/w_a \rfloor}$;
      \State replace $a^{-1}$ with $b^{\lfloor w_b/w_a \rfloor}a^{-1}$;
    \EndIf
  \EndWhile


\EndFunction
\end{algorithmic}
Now that the weight of $a$ in $w$ is zero, $w$ can be written as $b^{m_{1}a^{d_{1}}} \cdots b^{m_{r}a^{d_{r}}}$.
Additionally we normalize by conjugating $w$ by $a$ until $\mathrm{min} \{ d_{1}, \ldots ,d_{r} \} = 0$. Then it is possible to check if the relator is tidy, principal or monic by directly applying the definitions from the paragraph preceding Theorem \ref{CGW_theorem}.

\section{Bi-orderable and non-bi-orderable knot groups}

All group presentations were initially calculated using SnapPy by Nathan Dunfield, the presentations were then changed using the algorithm of Appendix A in order to apply Theorems \ref{CGW_theorem}(1) and \ref{CGW_theorem}(2).
\subsection{Fibered knots}
For fibered knots, our program found the same results as \cite{CR12}, with two exceptions which we believe came about due to rounding error in the numerical method used in \cite{CR12} to solve for the roots of the Alexander polynomial.  The two exceptions are:
\begin{enumerate}
\item In that paper, the knot $12n_{0019}$ is listed having a non-bi-orderable group, when in fact bi-orderability of its group cannot be determined by known theorems since its Alexander polynomial has both positive and negative real roots (and it is not a `special' polynomial, as defined in \cite{LRR08}).
\item The knot $12a_{0477}$ is fibered, and its Alexander polynomial is $\Delta(t) =1-11t+ 41t^2-63t^3+ 41t^4-11t^5+ t^6$, which has all real positive roots.  Therefore it has a bi-orderable knot group by the main theorem of \cite{CR12}, though it is listed as non-bi-orderable there.
\end{enumerate}

\subsection{Non-fibered knots}

The following $5$ nonfibered knots have bi-orderable groups, by applying either Theorem \ref{CGW_theorem}(2) or Theorem \ref{CGW_theorem}(3):
$6_{1}, \ 8_{1}, \ 10_{1}, \ 10_{13}, \ 12a_{803}$.

The following $79$ nonfibered knots are two-bridge and their Alexander polynomials have no positive real roots, and so their groups are not bi-orderable:
$5_{2}$, $\ 7_{2}$, $\ 7_{3}$, $\ 7_{4}$, $\ 7_{5}$, $\ 8_{8}$, $\ 8_{13}$, $\ 9_{2}$, $\ 9_{3}$, $\ 9_{4}$, $\ 9_{5}$, $\ 9_{6}$, $\ 9_{7}$, $\ 9_{9}$, $\ 9_{10}$, $\ 9_{13}$, $\ 9_{14}$, $\ 9_{18}$, $\ 9_{19}$, $\ 9_{23}$, $\ 10_{10}$, $\ 10_{12}$, $\ 10_{15}$, $\ 10_{19}$, $\ 10_{23}$, $\ 10_{27}$, $\ 10_{28}$, $\ 10_{31}$, $\ 10_{33}$, $\ 10_{34}$, $\ 10_{37}$, $\ 10_{40}$, $\ 11a_{13}$, $\ 11a_{75}$, $\ 11a_{77}$, $\ 11a_{85}$, $\ 11a_{89}$, $\ 11a_{90}$, $\ 11a_{95}$, $\ 11a_{98}$, $\ 11a_{111}$, $\ 11a_{119}$, $\ 11a_{178}$, $\ 11a_{183}$, $\ 11a_{186}$, $\ 11a_{188}$, $\ 11a_{191}$, $\ 11a_{192}$, $\ 11a_{193}$, $\ 11a_{195}$, $\ 11a_{205}$, $\ 11a_{210}$, $\ 11a_{234}$, $\ 11a_{235}$, $\ 11a_{236}$, $\ 11a_{238}$, $\ 11a_{242}$, $\ 11a_{243}$, $\ 11a_{246}$, $\ 11a_{247}$, $\ 11a_{307}$, $\ 11a_{333}$, $\ 11a_{334}$, $\ 11a_{335}$, $\ 11a_{336}$, $\ 11a_{337}$, $\ 11a_{339}$, $\ 11a_{341}$, $\ 11a_{342}$, $\ 11a_{343}$, $\ 11a_{355}$, $\ 11a_{356}$, $\ 11a_{357},\ 11a_{358}$, $\ 11a_{359}$, $\ 11a_{360}$, $\ 11a_{363}$, $\ 11a_{364}$, $\ 11a_{365}$.
 
The following $15$ knots have bridge index $3$, admit a two-generator presentation with a single tidy relator, and their Alexander polynomials have no positive real roots. Therefore their groups are not bi-orderable:
$9_{16}, \ 10_{52}, \ 10_{57}, \ 10_{128}, \ 10_{129}$, $\ 10_{130}, \ 10_{134}, \ 10_{135},\ 11a_{12}, \ 11a_{32}$, $\ 11a_{46}, \ 11a_{241}, \ 11a_{258}, \ 11n_{18}, \ 11n_{62}$.

The following $92$ knots admit a two-generator presentation with a single tidy relator and their Alexander polynomials have no positive real roots, so their groups are not bi-orderable (we do not know their bridge index, as that information is not listed on \textit{Knotinfo} for knots with twelve or more crossings):
$12a_{9}$, $\ 12a_{31}$, $\ 12a_{32}$, $\ 12a_{42}$, $\ 12a_{81}$, $\ 12a_{96}$, $\ 12a_{143}$, $\ 12a_{147}$, $\ 12a_{148}$, $\ 12a_{151}$, $\ 12a_{169}$, $\ 12a_{212}$, $\ 12a_{235}$, $\ 12a_{241}$, $\ 12a_{247}$, $\ 12a_{251}$, $\ 12a_{302}$, $\ 12a_{378}$, $\ 12a_{379}$, $\ 12a_{424}$, $\ 12a_{511}$, $\ 12a_{514}$, $\ 12a_{534}$, $\ 12a_{537}$, $\ 12a_{580}$, $\ 12a_{581}$, $\ 12a_{582}$, $\ 12a_{595}$, $\ 12a_{596}$, $\ 12a_{643}$, $\ 12a_{669}$, $\ 12a_{718}$, $\ 12a_{720}$, $\ 12a_{728}$, $\ 12a_{732}$, $\ 12a_{744}$, $\ 12a_{759}$, $\ 12a_{760}$, $\ 12a_{761}$, $\ 12a_{774}$, $\ 12a_{791}$, $\ 12a_{792}$, $\ 12a_{826}$, $\ 12a_{827}$, $\ 12a_{836}$, $\ 12a_{876}$, $\ 12a_{879}$, $\ 12a_{880}$, $\ 12a_{882}$, $\ 12a_{1029}$, $\ 12a_{1030}$, $\ 12a_{1033}$, $\ 12a_{1034}$, $\ 12a_{1129}$, $\ 12a_{1130}$, $\ 12a_{1132}$, $\ 12a_{1133}$, $\ 12a_{1138}$, $\ 12a_{1139}$, $\ 12n_{46}$, $\ 12n_{78}$, $\ 12n_{153}$, $\ 12n_{154}$, $\ 12n_{166}$, $\ 12n_{167}$, $\ 12n_{169}$, $\ 12n_{170}$, $\ 12n_{200}$, $\ 12n_{236}$, $\ 12n_{239}$, $\ 12n_{241}$, $\ 12n_{243}$, $\ 12n_{244}$, $\ 12n_{248}$, $\ 12n_{250}$, $\ 12n_{251}$, $\ 12n_{288}$, $\ 12n_{289}$, $\ 12n_{305}$, $\ 12n_{307}$, $\ 12n_{308}$, $\ 12n_{310}$, $\ 12n_{374}$, $\ 12n_{404}$, $\ 12n_{501}$, $\ 12n_{502}$, $\ 12n_{503}$, $\ 12n_{575}$, $\ 12n_{594}$, $\ 12n_{650}$, $\ 12n_{723}$, $\ 12n_{851}$.

\bibliographystyle{plain}

\bibliography{ordbook}

\def\cprime{$'$}
\begin{thebibliography}{10}

\bibitem{CL}
J.~C. Cha and C.~Livingston.
\newblock Knotinfo: Table of knot invariants.
\newblock http://www.indiana.edu/knotinfo.

\bibitem{CGW14}
I.~M. Chiswell, A.~M.~W. Glass, and John~S. Wilson.
\newblock Residual nilpotence and ordering in one-relator groups and knot
  groups.
\newblock Preprint, available via http://arxiv.org/abs/1405.0994.

\bibitem{CR12}
Adam Clay and Dale Rolfsen.
\newblock Ordered groups, eigenvalues, knots, surgery and {$L$}-spaces.
\newblock {\em Math. Proc. Cambridge Philos. Soc.}, 152(1):115--129, 2012.

\bibitem{SnapPy}
Marc Culler, Nathan~M. Dunfield, and Jeffrey~R. Weeks.
\newblock Snap{P}y, a computer program for studying the topology of
  $3$-manifolds.
\newblock Available at \url{http://snappy.computop.org} (09/10/2014).

\bibitem{HJ13}
Jim Hoste and Patrick~D. Shanahan.
\newblock Twisted {A}lexander polynomials of 2-bridge knots.
\newblock {\em J. Knot Theory Ramifications}, 22(1):1250138, 29, 2013.

\bibitem{Ito13a}
Tetsuya Ito.
\newblock A remark on the {A}lexander polynomial criterion for the
  bi-orderability of fibered 3-manifold groups.
\newblock {\em Int. Math. Res. Not. IMRN}, (1):156--169, 2013.

\bibitem{LRR08}
Peter~A. Linnell, Akbar~H. Rhemtulla, and Dale P.~O. Rolfsen.
\newblock Invariant group orderings and {G}alois conjugates.
\newblock {\em Journal of Algebra}, 319(12):4891--4898, 2008.

\bibitem{LS77}
Roger~C. Lyndon and Paul~E. Schupp.
\newblock {\em Combinatorial group theory}.
\newblock Springer-Verlag, Berlin, 1977.
\newblock Ergebnisse der Mathematik und ihrer Grenzgebiete, Band 89.

\bibitem{Murasugi61}
Kunio Murasugi.
\newblock Remarks on knots with two bridges.
\newblock {\em Proc. Japan Acad.}, 37:294--297, 1961.

\bibitem{NR14}
Geoff Naylor and Dale Rolfsen.
\newblock Generalized torsion in knot groups.
\newblock Preprint, available via http://arxiv.org/abs/1409.5730.

\bibitem{PR03}
Bernard Perron and Dale Rolfsen.
\newblock On orderability of fibred knot groups.
\newblock {\em Math. Proc. Cambridge Philos. Soc.}, 135(1):147--153, 2003.

\bibitem{PR06}
Bernard Perron and Dale Rolfsen.
\newblock Invariant ordering of surface groups and 3-manifolds which fibre over
  {$S\sp 1$}.
\newblock {\em Math. Proc. Cambridge Philos. Soc.}, 141(2):273--280, 2006.

\bibitem{SN99}
Nikolai Saveliev.
\newblock {\em Lectures on the topology of {$3$}-manifolds}.
\newblock de Gruyter Textbook. Walter de Gruyter \& Co., Berlin, 1999.
\newblock An introduction to the Casson invariant.

\bibitem{Schubert56}
Horst Schubert.
\newblock Knoten mit zwei {B}r\"ucken.
\newblock {\em Math. Z.}, 65:133--170, 1956.

\end{thebibliography}

\end{document}